\documentclass[letterpaper, 10 pt, conference]{ieeeconf}  

\IEEEoverridecommandlockouts                              

\usepackage{afterpage}
\usepackage{epsfig} 
\usepackage{times} 
\usepackage{amsmath} 
\usepackage{amssymb}  
\usepackage{amsfonts}
\usepackage{mathptmx} 
\usepackage{tabularx}
\usepackage{epstopdf}
\usepackage{subfigure}
\usepackage{float}
\usepackage{epic,color}
\usepackage{mathrsfs}
\usepackage{dsfont,MnSymbol}
\usepackage{algorithm}
\usepackage{algorithmic}
\usepackage{multirow} 
\usepackage[T1]{fontenc}
\newcounter{rmnum}
\newenvironment{romannum}{\begin{list}{{\upshape (\roman{rmnum})}}{\usecounter{rmnum}
\setlength{\leftmargin}{14pt}
\setlength{\rightmargin}{8pt}
\setlength{\itemsep}{2pt}
\setlength{\itemindent}{-1pt}
}}{\end{list}}

\newcounter{anum}

\newlength{\noteWidth}
\long\def\notes#1{\ifinner
             {\tiny #1}
             \else
              \marginpar{\parbox[t]{\noteWidth}{\raggedright\tiny #1}}
               \fi}

\def\IEEEQEDclosed{\mbox{\rule[0pt]{1.3ex}{1.3ex}}}

\def\qed{\ifmmode\IEEEQEDclosed\else{\unskip\nobreak\hfil
\penalty50\hskip1em\null\nobreak\hfil\IEEEQEDclosed
\parfillskip=0pt\finalhyphendemerits=0\endgraf}\fi}

\def\qed{\hspace*{\fill}~\IEEEQED\par\endtrivlist\unskip}

\def\Re{\mathbb{R}}
\def\R{\mathbb{R}}

\def\ind{\text{\rm\large 1}}
\def\varble{\,\cdot\,}

\def\varble{\,\cdot\,}

\def\Sec#1{Sec.~\ref{#1}}
\def\Fig#1{Fig.~\ref{#1}}

\def\notes#1{\marginpar{\tiny #1}\typeout{Notes!
Notes!
Notes!
}}
\renewcommand{\notes}[1]{\typeout{notes!}}
\def\FRAC#1#2#3{\genfrac{}{}{}{#1}{#2}{#3}}
\def\half{{\mathchoice{\FRAC{1}{1}{2}}%
{\FRAC{2}{1}{2}}%
{\FRAC{3}{1}{2}}%
{\FRAC{4}{1}{2}}}}

\def\ind{\bbbone}

\def\Re{\field{R}}
\def\k{{\sf K}}

\def\Sec#1{Sec.~\ref{#1}}

\def\clZ{{\cal Z}}

\def\Sec#1{Sec~\ref{#1}}

\def\E{{\sf E}}

\def\Prob{{\sf P}}

\def\R{\mathbb{R}}

\def\Fig#1{Fig.~\ref{#1}}
\def\Sec#1{Sec.~\ref{#1}}

\def\IEEEQEDclosed{\mbox{\rule[0pt]{1.3ex}{1.3ex}}}
\def\qed{\nobreak\hfill\IEEEQEDclosed}

\def\clZ{{\cal Z}}
\def\varble{\,\cdot\,}

\newtheorem{theorem}{Theorem}
\newtheorem{example}{Example}

\newtheorem{lemma}{Lemma}
\newtheorem{remark}{Remark}
\newtheorem{proposition}{Proposition}

\def\beq{\begin{eqnarray}} 
\def\bc{\begin{center}} 
\def\be{\begin{enumerate}}
\def\bi{\begin{itemize}} 
\def\bs{\begin{small}}
\def\bS{\begin{slide}}
\def\ec{\end{center}} 
\def\ee{\end{enumerate}}
\def\ei{\end{itemize}}
\def\es{\end{small}}
\def\eS{\end{slide}}
\def\eeq{\end{eqnarray}}

\newcommand{\newP}[1]{\medskip\noindent{\bf #1:}}

\newcommand{\PP}{{\sf P}}

\newcommand{\ud}{\,\mathrm{d}}

\def\Re{\mathbb{R}}
\def\E{{\sf E}}

\def\ind{\text{\rm\large 1}}
\def\varble{\,\cdot\,}

\def\Sec#1{Sec.~\ref{#1}}

\def\clZ{{\cal Z}}

\renewcommand{\Re}{\mathbb{R}}

\def\Prob{{\sf P}}
\def\FRAC#1#2#3{\genfrac{}{}{}{#1}{#2}{#3}}

\newcommand{\kepsnorm}{{k}^{(\epsilon)}}
\newcommand{\kepsnormN}{{k}^{(\epsilon,N)}}

\newcommand{\keps}{{k}^{(\epsilon)}}

\newcommand{\phiepsN}{{\phi}^{(\epsilon,N)}}
\newcommand{\phieps}{{\phi}^{(\epsilon)}}

\newcommand{\HH}{H^1_0}

\newcommand{\phiM}{\phi^{(M)}}
\newcommand{\phiMN}{\phi^{(M,N)}}

\newcommand{\rhoeps}{\rho^{(\epsilon)}}
\newcommand{\Teps}{{T}^{(\epsilon)}}

\newcommand{\TepsN}{{T}^{(\epsilon,N)}}
\newcommand{\mueps}{{\mu}^{(\epsilon)}}
\newcommand{\geps}{{g}^{(\epsilon)}}
\newcommand{\neps}{{n}^{(\epsilon)}}
\newcommand{\Ltwoeps}{{L^2_0(\mueps)}}
\newcommand{\FF}{C_0(\Omega)}
\newcommand{\infnorm}[1]{\|#1\|_\infty}
\newcommand{\PhiepsN}{{\Phi}^{(\epsilon,N)}}

\title{\LARGE \bf
Gain Function Approximation in the Feedback
Particle Filter}

\author{Amirhossein Taghvaei, Prashant G. Mehta 
\thanks{Financial support from the NSF CMMI grants 1334987 and 1462773 is gratefully acknowledged. 
}
\thanks{A.~Taghvaei and P.~G.~Mehta are with the Coordinated
  Science Laboratory and the Department of Mechanical Science and
  Engineering at the University of Illinois at Urbana-Champaign (UIUC).
{\tt\scriptsize taghvae2@illinois.edu; mehtapg@illinois.edu;}}
}

\begin{document}

\maketitle
\thispagestyle{empty}
\pagestyle{empty}

\begin{abstract}
This paper is concerned with numerical algorithms for gain function
approximation in the feedback particle filter.  The exact gain
function is the solution of a Poisson equation involving a probability-weighted
Laplacian.  The problem is to approximate this solution using {\em only} 
particles sampled from the probability distribution.  
Two algorithms are presented: a Galerkin algorithm and a kernel-based
algorithm.  Both the algorithms are adapted to the samples and
do not require approximation of the probability distribution as an intermediate step.  
The paper contains error analysis for the algorithms as
well as some comparative numerical results for a non-Gaussian
distribution.  These algorithms are also applied and illustrated for a
simple nonlinear filtering example.      

\end{abstract}
\section{Introduction}

This paper is concerned with algorithms for numerically approximating
the solution of a certain linear partial differential equation (pde) that arises in the problem of
nonlinear filtering.  In continuous time, the filtering problem
pertains to the following stochastic differential equations (sdes):     
\begin{subequations}
\begin{align}
\ud X_t &= a(X_t)\ud t + \ud B_t,
\label{eqn:Signal_Process}
\\
\ud Z_t &= h(X_t)\ud t + \ud W_t,
\label{eqn:Obs_Process}
\end{align}
\end{subequations}
where $X_t\in\Re^d$ is the (hidden) state at time $t$, $Z_t \in\Re$ is the
observation, and $\{B_t\}$, $\{W_t\}$ are two mutually independent
standard Wiener processes taking values in $\Re^d$ and $\Re$, respectively. The mappings
$a(\cdot): \Re^d \rightarrow \Re^d$ and $h(\cdot): \Re^d \rightarrow
\Re$ are $C^1$ functions.  
Unless noted otherwise, all probability distributions are assumed to
be absolutely continuous with respect to the Lebesgue measure, and
therefore will be identified with their densities.  The choice of
observation being scalar-valued ($Z_t\in\Re$) is made for notational ease.

The objective of the filtering problem is to estimate the
posterior distribution of $X_t$ given the time history of observations
(filtration) $\clZ_t :=
\sigma(Z_s:  0\le s \le t)$. The density of the posterior
distribution is denoted by $p^*$, so
that for any measurable set $A\subset \Re^d$,
\begin{equation}
\int_{x\in A} p^*(x,t)\, \ud x   = \Prob [X_t \in A\mid \clZ_t ].
\nonumber
\end{equation}
The filter is infinite-dimensional since it
defines the evolution, in  the space of probability measures,
of $\{p^*(\varble ,t) : t\ge 0\}$.  If $a(\varble)$, $h(\varble)$ are linear functions, the solution is given by the
finite-dimensional Kalman-Bucy filter.  The article~\cite{budchelee07} surveys numerical methods to
approximate the nonlinear filter. One approach described in this survey is particle filtering.

The particle filter is a simulation-based algorithm to approximate the
filtering task~\cite{smith2013}. The key step is the
construction of $N$ stochastic processes $\{X^i_t : 1\le i \le N\}$:
The value $X^i_t \in \Re^d$ is the state for the $i$-th
particle at time $t$. For each time $t$, the empirical distribution
formed by the particle population is used to approximate the
posterior distribution.  Recall that this is  defined for any measurable set $A\subset\Re^d$ by,
\begin{equation}
p^{(N)}(A,t) = \frac{1}{N}\sum_{i=1}^N \ind [ X^i_t\in A].\nonumber
\end{equation}
A common approach in particle filtering is called {\em
sequential importance sampling}, where particles are generated
according to their importance weight at every time
step~\cite{bain2009,smith2013}.

In our earlier papers~\cite{taoyang_acc11, taoyang_cdc11,
  taoyang_TAC12}, an alternative feedback control-based approach to
the construction of a particle filter was introduced. 
The resulting particle filter, referred to as the feedback particle
filter (FPF), is a controlled system.  The dynamics of the
$i$-th particle have the following gain feedback form,
\begin{equation}
\ud X^i_t = a(X^i_t) \ud t + \ud B^i_t + \k_t(X^i_t) \circ (\ud Z_t -
\frac{h(X^i_t) + \hat{h}_t}{2}\ud t),
\label{eqn:particle_filter_nonlin_intro}
\end{equation}  
where $\{B^i_t\}$ are mutually independent standard Wiener processes and
$\hat{h}_t := \E[h(X_t^i)|\mathcal{Z}_t]$. The initial condition $X^i_0$ 
is drawn from the initial density $p^*(x,0)$ of $X_0$ independent
of $\{B^i_t\}$.  Both  $\{B^i_t\}$ and $\{X^i_0\}$
are also assumed to be independent of $X_t,Z_t$.  The $\circ$
indicates that the sde is expressed in its Stratonovich form. 

The gain function $\k_t$ is obtained by solving a weighted Poisson
equation: For each fixed time $t$, the function $\phi$ is the solution to a Poisson equation,
\begin{flalign}
\label{eqn:EL_phi_intro}
\text{BVP} && &
\begin{aligned}
\nabla \cdot (p(x,t) \nabla \phi(x,t) ) & = - (h(x)-\hat{h}) p(x,t),\\
\int \phi(x,t) p(x,t) \ud x & = 0 \qquad \text{(zero-mean)},
\end{aligned} &
\end{flalign}
for all $x\in \Re^d$ where $\nabla$ and $\nabla \cdot $ denote the
gradient and the divergence operators, respectively, and
$p$ denotes the conditional density of $X_t^i$ given
$\mathcal{Z}_t$.  

In terms of the solution $\phi$, the gain function is given by,
\begin{equation*}
\k_t(x) = \nabla \phi(x,t)\, .
\label{eqn:gradient_gain_fn_intro}
\end{equation*}
Note that the gain function $\k_t$ is vector-valued (with dimension
$d\times 1$) and it needs to be obtained for each fixed time $t$.
For the linear Gaussian case, the gain function
is the Kalman gain.  For the general nonlinear non-Gaussian problem, the FPF~\eqref{eqn:particle_filter_nonlin_intro}
is exact, given an exact gain function and an exact initialization
$p(\varble,0)=p^*(\varble,0)$. Consequently, if the initial conditions
$\{X^i_0\}_{i=1}^N$  are drawn from the initial density
$p^*(\varble,0)$ of $X_0$, then, as $N\rightarrow\infty$, the
empirical distribution of the particle system approximates the
posterior density $p^*(\varble,t)$ for each $t$.

A numerical implementation of the
FPF~\eqref{eqn:particle_filter_nonlin_intro} requires a numerical
approximation of the gain function $\k_t$ and the mean $\hat{h}_t$ at
each time-step.  The mean is approximated empirically, $\hat{h}_t \approx
\frac{1}{N}\sum_{i=1}^N h(X_t^i) =: \hat{h}_t^{(N)}$.  The gain function
approximation - the focus of this paper - is a challenging problem because of two reasons:  i) Apart
from the linear Gaussian case, there are no known closed-form
solutions of~\eqref{eqn:EL_phi_intro}; ii)
The density $p(x,t)$ is not explicitly known.  At each time-step, one only has
samples $X^i_t$.  These are assumed to be i.i.d sampled from $p$.     
Apart from the FPF algorithm, solution of the Poisson equation is also
central to a number of other algorithms for nonlinear filtering~\cite{daum10,yang_discrete}.

In our prior work, we have obtained results on
existence, uniqueness and regularity of the solution to the Poisson
equation, based on certain weak formulation of the Poisson equation.
The weak formulation led to a Galerkin numerical algorithm.  
The main limitation of the Galerkin is that the algorithm requires a pre-defined set of basis
functions - which scales poorly with the dimension $d$ of the state.
The Galerkin algorithm can also exhibit certain numerical issues related to the
Gibb's phenomena.  This can in turn lead to numerical instabilities in
simulating the FPF.

The contributions of this paper are as follows:
We present a new basis-free kernel-based algorithm for approximating the
solution of the gain function.  The key step is to construct a Markov
matrix on a certain graph defined on the space of particles $\{X^i_t\}_{i=1}^N$.
The value of the function $\phi$ for the particles, $\phi(X^i_t)$, is then
approximated by solving a fixed-point problem involving
the Markov matrix.  The
fixed-point problem is shown to be a contraction and the method of
successive approximation applies to numerically obtain the solution. 

We present results on error analysis for both the Galerkin and
the kernel-based method.  These results are illustrated with the aid of an
example involving a multi-modal distribution.  Finally, the two
methods are compared for a filtering problem with a non-Gaussian
distribution.

In the remainder of this paper, we express the linear operator
in~\eqref{eqn:EL_phi_intro} as a weighted Laplacian $\Delta_\rho
\phi:=\frac{1}{\rho}\nabla \cdot (\rho\nabla \phi)$ where additional
assumptions on the density $\rho$ appear in the main body of the
paper.  In recent years, this operator and the associated Markov
semigroup have received considerable attention with several
applications including spectral clustering, dimensionality reduction,
supervised learning etc~\cite{coifman,hein2006}.  For a mathematical
treatment, see the monographs~\cite{grigor2006,bakry2013}.  Related specifically to
control theory, there are important connections with stochastic
stability of Markov operators~\cite{glynn96,meyn12}.

The outline of this paper is as follows: The mathematical
preliminaries appear in \Sec{sec:prelim}.  The Galerkin and the
kernel-based algorithms for the gain function approximation appear in
\Sec{sec:galerkin} and \Sec{sec:graph}, respectively.  The nonlinear
filtering example appears in \Sec{sec:numerics}.

\section{Mathematical Preliminaries}
\label{sec:prelim}

The Poisson equation~\eqref{eqn:EL_phi_intro} is expressed as,
\begin{flalign}
\label{eqn:EL_phi_prelim}
\text{BVP} && &
\begin{aligned}
-\Delta_\rho \phi & = h,\\
\int \phi \rho \ud x & = 0 \qquad \text{(zero-mean)},
\end{aligned} &
\end{flalign}
where $\rho$ is a probability density on $\Re^d$, $\Delta_\rho
\phi:=\frac{1}{\rho}\nabla \cdot (\rho\nabla \phi)$ and, without loss
of generality, it is assumed $\hat{h} = \int h \rho \ud x=0$.  

\medskip

\noindent\textbf{Problem statement:} Approximate
the solution $\phi(X^i)$ and $\nabla \phi(X^i)$ given $N$ independent
samples $X^i$ drawn from $\rho$.  The density $\rho$ is not explicitly
known.  

\medskip

For the problem to be well-posed requires definition of the function
spaces and additional assumptions on $\rho$ and $h$ enumerated
next: Throughout this paper, $\mu$ is an absolutely
continuous probability measure on $\Re^n$ with associated density $\rho$.
$L^2(\Re^d,\mu)$ is the Hilbert space of square integrable functions on
$\Re^d$ equipped with the inner-product,
\[
<\phi,\psi>:=\int\phi(x)\psi(x)\ud\mu(x).
\]
The associated norm is
denoted as $\|\phi\|^2_2:=<\phi,\phi>$. The space  
$H^1(\Re^d,\mu)$ is the
space of square integrable functions whose derivative (defined in the
weak sense) is in $L^2(\Re^d,\mu)$.  We use $L^2$ and $H^1$ to denote $L^2(\Re^d,\mu)$ and
$H^1(\Re^d,\mu)$, respectively.  For the zero-mean solution of interest, we
additionally define the co-dimension 1 subspace $L^2_0 :=\{\phi \in
L^2; \int \phi \ud \mu =0\}$ and $H^1_0 :=\{\phi \in
H^1; \int \phi \ud \mu =0\}$.  
$L^\infty$ is used to denote the
space of functions that are bounded a.e.\ (Lebesgue) and the sup-norm
of a function $\phi\in L^\infty$ is denoted as $\|\phi\|_{\infty}$. 


\medskip

The following assumptions are made throughout the paper:
\begin{romannum}
\item {\bf Assumption A1:} The probability density function $\rho$ is
  of the form $\rho(x)= e^{-(x-\mu)^T \Sigma^{-1} (x-\mu) - V(x)}$ where $\mu \in \Re^d$, $\Sigma$ is a
  positive-definite matrix and $V \in C^2$ with $V\in L^\infty$, and
  its derivatives $DV,D^2V \in
  L^\infty$. 
\item {\bf Assumption A2:} The function $h \in L^2$ and $\int h\ud\mu=0$. 
\end{romannum}
 
Under assumption A1, the density $\rho$ admits a spectral gap (or
Poincar\'e inequality)\cite{bakry2013}, i.e $\exists \;\lambda_1 >0$ such that,
\begin{equation}
\int \phi^2\rho \ud x \leq \frac{1}{\lambda_1}\int |\nabla \phi|^2 \rho \ud x, \quad 
\forall \phi \in \HH.
\label{eq:poincare}
\end{equation} 
Furthermore, the spectrum is known to be discrete with an ordered sequence of
eigenvalues $0=\lambda_0<\lambda_1\le\lambda_2\le\hdots$ and
associated eigenfunctions $\{e_n\}$ that form a complete orthonormal
basis of $L^2$ [Corollary~4.10.9 in ~\cite{bakry2013}].  The
trivial eigenvalue $\lambda_0=1$ with associated eigenfunction $e_0=1$.
On the subspace of zero-mean functions, the spectral decomposition
yields: For $\phi\in L^2_0$,
\begin{equation}
\label{eq:spectral_rep}
-\Delta_\rho \phi = \sum_{m=1}^\infty \lambda_m <e_m,\phi>e_m.
\end{equation}
The spectral gap condition~\eqref{eq:poincare} implies that
$\lambda_1>0$.  Consequently, the semigroup
\begin{equation}
e^{t\Delta_\rho}\phi := \sum_{m=1}^{\infty} e^{-t\lambda_m}<e_m,\phi>e_m
\label{eq:semigroup}
\end{equation} 
is a strict contraction on the subspace $L^2_0$.  It is also easy to see that
$\mu$ is an invariant measure and $\int e^{t\Delta_\rho}\phi(x)
\ud\mu(x) = \int \phi(x) \ud\mu(x) = 0$ for all $\phi\in L^2_0$. 

\medskip

\begin{example}
If the density $\rho$ is Gaussian with mean $\mu\in\mathbb{R}^d$ and a
positive-definite covariance matrix $\Sigma$, the spectral gap
constant (1/$\lambda_1$) equals the largest eigenvalue of $\Sigma$.  The
eigenfunctions are the Hermite functions.  Given a linear function
$h(x) = H \cdot x$, the unique solution of the BVP~\eqref{eqn:EL_phi_prelim} is given by,
\begin{equation*}
\phi(x) = (\Sigma H) \cdot (x - \mu),
\end{equation*}
where $\cdot$ denotes the vector dot
product in $\mathbb{R}^d$. 
In this case, $\k = \nabla\phi = \Sigma H$ is the Kalman gain. \qed


\end{example}
\medskip

\begin{example} In the scalar case (where $d=1$), the Poisson equation
  is:
\begin{equation*}
-\frac{1}{\rho(x)}\frac{\ud }{\ud x}(\rho(x) \frac{\ud \phi}{\ud
  x}(x)) = h.
\end{equation*}
Integrating twice yields the solution explicitly,
\begin{equation}
\begin{aligned}
\frac{\ud \phi}{\ud x}(x) &= -\frac{1}{\rho(x)}\int_{-\infty}^x \rho(z)h(z)\ud z  \\
\phi(x) &= -\int_{-\infty}^{x}\frac{\ud y}{\rho(y)}\int_{-\infty}^y \rho(z)h(z)\ud z  
\end{aligned}
\label{eq:scalar}
\end{equation}

For the particular choice of $\rho$ as the sum of two Gaussians
$N(-1,\sigma^2)$ and $N(+1,\sigma^2)$ with $\sigma=0.4$ and $h(x)=x$, the solution
obtained using~\eqref{eq:scalar} is depicted in \Fig{fig:truesol}.
Since $\frac{\ud h}{\ud x}>0$, the positivity of the gain function $\frac{\ud \phi}{\ud
  x}(x)$ follows from the maximum principle for elliptic pdes~\cite{EVANS_PDE}. 
\qed

\begin{figure}
\centering
\includegraphics[width=\columnwidth]{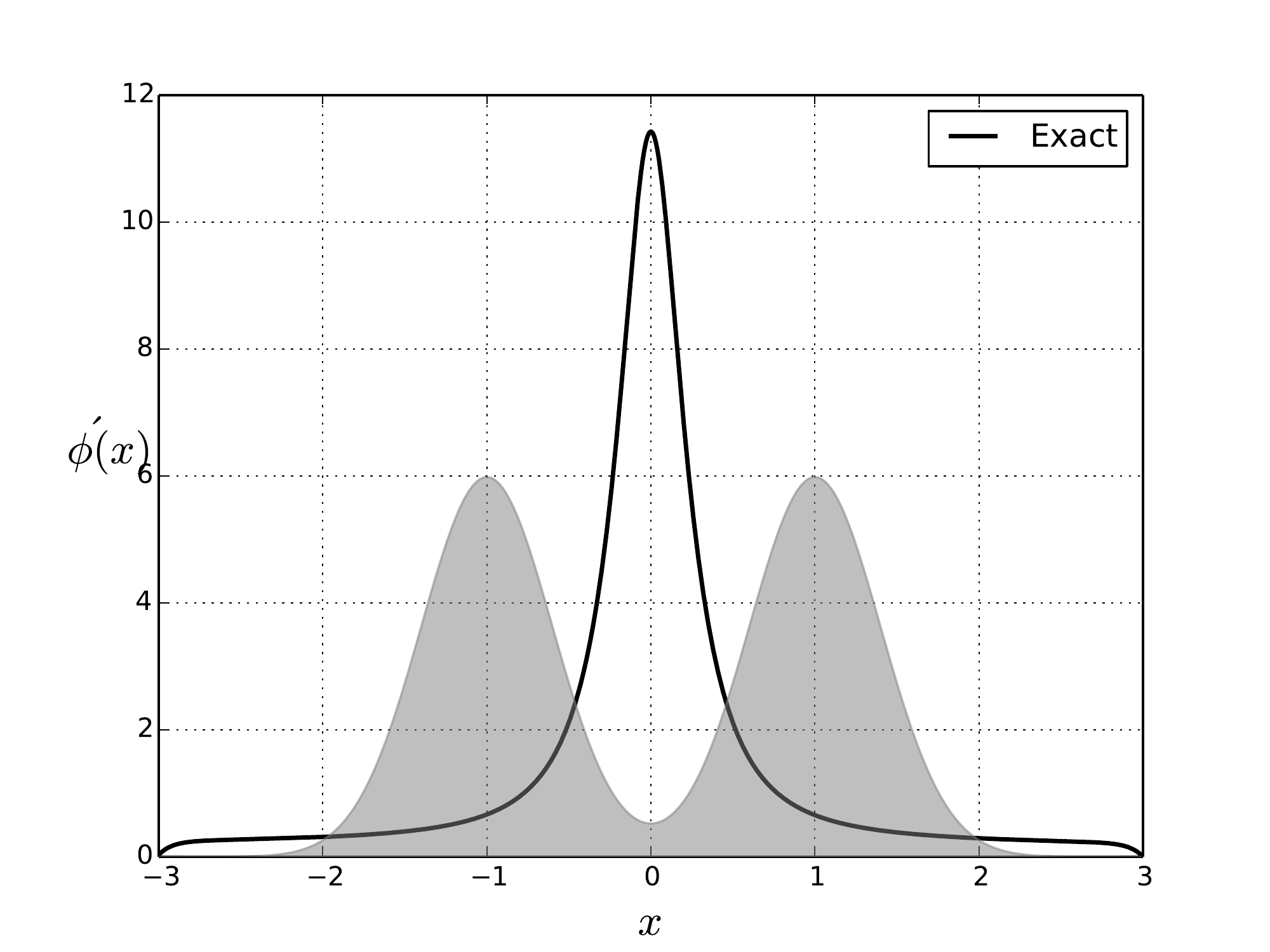}
\caption{The exact solution to the Poisson equation using the
  formula~\eqref{eq:scalar}.  The density $\rho$ is the sum of two Gaussians
$N(-1,\sigma^2)$ and $N(+1,\sigma^2)$, and $h(x)=x$.  The density is depicted as the shaded curve in the background.}

\label{fig:truesol}
\end{figure}
  
\label{example:truesol}
\end{example}
\section{Galerkin for Gain Function Approximation}
\label{sec:galerkin}

\newP{Weak formulation} A function $\phi \in H_0^1$ is said to be a weak solution
of Poisson's equation~\eqref{eqn:EL_phi_prelim} if
\begin{equation}
<\nabla \phi,\nabla \psi> = <h,\psi>, \quad
\forall\; \psi\in H_0^1.\label{eqn:EL_phi_weak}
\end{equation}
It is shown in~\cite{poisson15} that, under
Assumptions~(A1)-(A2), there exists a unique weak solution of the
Poisson equation.

The Galerkin approximation involves
solving~\eqref{eqn:EL_phi_weak} in a finite-dimensional subspace
$S\subset H^1_0(\R^d;\rho)$.  The solution $\phi$ is approximated as,
\begin{equation}
\phiM(x) = \sum_{m=1}^M c_m \psi_m(x),
\label{eq:phiM}
\end{equation}   
where $\{\psi_m(x) \}_{m=1}^M$ are a {\em given} set of basis
functions.

The finite-dimensional approximation of~\eqref{eqn:EL_phi_weak} is to choose
constants $\{c_m\}_{m=1}^M$ such that
\begin{equation}
<\nabla \phi^{(M)},\nabla \psi> = <h,\psi>, \;\;
\forall \; \psi\in S,
\label{eqn:BVP_phi_weak_expect_fd}
\end{equation}
where $S:=\text{span}\{\psi_1,\hdots,\psi_M\}\subset H^1_0$.

Denoting $[A]_{ml} = <\nabla \psi_l \cdot \nabla \psi_m>$, $b_m
= <h, \psi_m>$, and $c=(c_1,c_2,\hdots,c_M)^T$, the finite-dimensional
approximation~\eqref{eqn:BVP_phi_weak_expect_fd} is expressed as a linear matrix
equation:
\begin{equation}
    Ac= b.\label{eqn:linmatrixeqn}
\end{equation}

In a numerical implementation, the matrix $A$ and vector $b$ are approximated as,
\begin{align}
    [A]_{ml} & = <\nabla \psi_l \cdot \nabla \psi_m> \approx
    \frac{1}{N} \sum_{i=1}^N \nabla \psi_l (X^i) \cdot
   \nabla \psi_m (X^i)=: [A]_{ml}^{(N)}, 
   \label{formula:A_ml_sample} \\
    b_m & = <h, \psi_m> \approx \frac{1}{N} \sum_{i=1}^N
    h(X^i) \psi_m(X^i)=:b_m^{(N)}. 
   \label{formula:b_m_sample}
\end{align}
The resulting solution of the matrix
equation~\eqref{eqn:linmatrixeqn}, with $A=A^{(N)}$ and $b=b^{(N)}$,
is denoted as $c^{(N)}$,
\begin{equation}
    A^{(N)} c^{(N)}= b^{(N)}.\label{eqn:linmatrixeqnN}
\end{equation}
Using~\eqref{eq:phiM}, we obtain the
particle-based approximation of the solution:
\begin{equation}
\phi^{(M,N)}(x) = \sum_{m=1}^M c_m^{(N)} \psi_m(x).
\label{eq:phiMN}
\end{equation}  
In terms of this solution, the gain function is obtained as,
\[
\nabla\phi^{M,N}(x) = \sum_{m=1}^M c_m^{(N)} \nabla \psi_m(x).
\]

\newP{Convergence analysis}  The following is a summary of the 
approximations in the Galerkin algorithm:
\begin{align*}
\text{Exact}: &~& <\nabla \phi,\nabla \psi> &= <h,\psi>, \quad
\forall \psi\in H_0^1\\
\text{Galerkin approx:} &~& <\nabla \phi^{(M)},\nabla \psi> &= <h,\psi>, \;
\forall \psi\in S\subset H_0^1
\end{align*}
Empirical approximation:
\begin{align*}
\frac{1}{N} \sum_{i=1}^N \nabla \phi^{(M,N)}(X^i)
\cdot \nabla \psi(X^i) = \frac{1}{N} \sum_{i=1}^N h(X^i)\psi(X^i),\;\;
\forall \psi\in S
\end{align*}
We are interested in the error analysis of these approximations as a
function of both $M$ and $N$.  Note that the approximation error
$\phi^{(M)}-\phi^{(M,N)}$ is random because $X^i$ are sampled randomly
from the probability distribution $\mu$.  The following Proposition
provides error bounds for the case where the basis functions are the
eigenfunctions of the Laplacian.  The proof appears in the Appendix~\ref{app:app1}.

\begin{proposition} Consider the empirical Galerkin approximation of the Poisson
  equation~\eqref{eqn:EL_phi_weak} on the space
  $S:=\text{span}\{e_1,e_2,\hdots,e_M\}$ of the first $M$
  eigenfunctions.  Fix $M<\infty$.  Then the
  solution for the matrix equation~\eqref{eqn:linmatrixeqnN} exists with probability
  approaching 1 as $N\rightarrow\infty$.
  And there is a sequence of random variables $\{\epsilon_N\}$ such that  
\begin{equation}
\|\phiMN - \phi\|_2\le
\frac{1}{\lambda_M^2}\|h\|^2_{2} + \epsilon_N,
\end{equation}
where $\epsilon_N\rightarrow 0$ as $N\rightarrow\infty$ a.s.
\qed
\label{prop:galerkin}
\end{proposition}

\medskip

\begin{remark}
In practice, the eigenfunctions of the Laplacian are not known.  The
basis functions are typically picked from the polynomial family, e.g.,
the Hermite functions.  In this case, the bounds will provide
qualitative assessment of the error provided the eigenfunctions
associated with the first $M$ eigenvalues are `approximately' in $S$.
Quantitatively, the additional error may be bounded in terms of the
projection error between the eigenfunction and the projection onto
$S$.      \qed
\end{remark}

\medskip
\begin{example}
We next revisit the bimodal distribution first introduced in the
Example~\ref{example:truesol}. \Fig{fig:galerkin} depicts the
empirical Galerkin approximation, with $N=200$ samples, of the gain
function.  The basis functions are from the polynomial family -
$\{x,x^2,\hdots,x^M\}$ - and the figure depicts the gain functions with
$M=1,3,5$ modes.  The $(M=1)$ case is referred to as the {\em constant-gain
approximation} where,
\[
\frac{\ud\phi}{\ud x}(x) = \int h(x) \, x \; \ud \mu(x) = (\text{const.})
\]  
For the linear Gaussian case, the (const.) is the Kalman gain.  The
plots included in the Figure demonstrate the Gibb's phenomena as more
and more modes are included in the Galerkin.  Particularly concerning
is the change in the sign of the gain which leads to negative
values of gain for some particles contradicting the positivity
property of the gain described in Example~\ref{example:truesol}. 
As discussed in the filtering example in \Sec{sec:numerics}, the
negative values of the gain can cause numerical issues and lead to
erroneous results for the filter. 
\qed
\label{example:Galerkin}
\end{example}

\begin{figure}
\centering
\begin{tabular}{c}
\subfigure{
\includegraphics[width=\columnwidth]{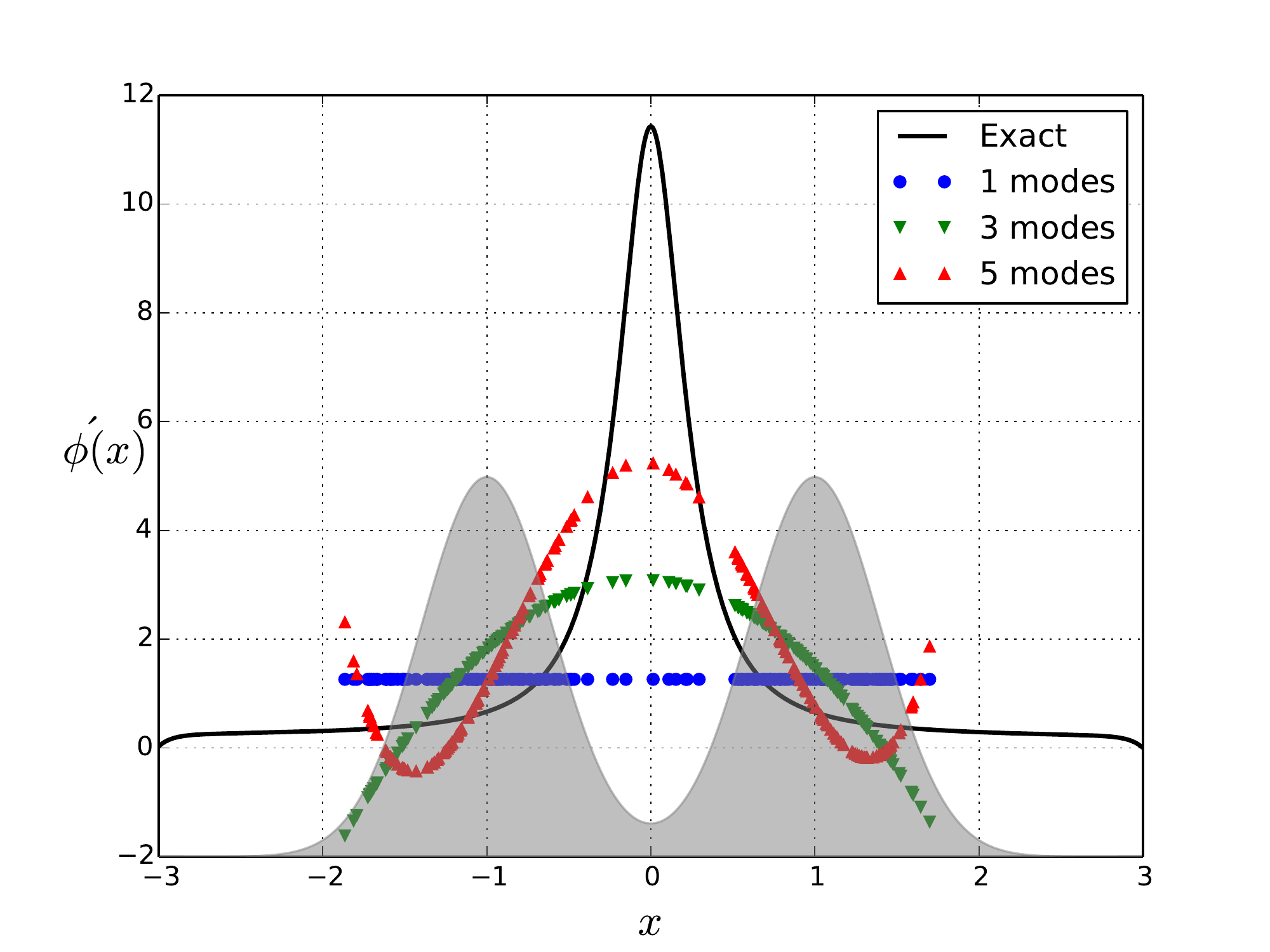}
}
\end{tabular}
\caption{Comparison of the exact solution and its empirical Galerkin
  approximation with $M=1,3,5$ modes and $N=200$ particles.  The density is depicted as the shaded curve in the background.}
\label{fig:galerkin}
\end{figure}

\section{Semigroup Approximation of the Gain }
\label{sec:graph}

\newP{Semigroup formulation} The semigroup
formula~\eqref{eq:semigroup} is used to construct the solution of the
Poisson equation by solving the following fixed-point equation for
{\em any} fixed positive value of $t$:
\begin{equation}
\phi = e^{t\Delta_\rho} \phi+ \int_0^t e^{s\Delta_\rho} h \ud s.
\label{eq:fixed1}
\end{equation}
A unique solution exists because $e^{t\Delta_\rho}$ is a contraction
on $L^2_0$.  

The approximation proposed in this section involves approximating the
semigroup as a perturbed integral operator, for small positive values of
$t=\epsilon$.  The following approximation
of the semigroup appears in~\cite{coifman,hein2006}:
\begin{equation}
\Teps \phi (x) := \frac{\int \kepsnorm(x,y)\phi(y) \ud\mu(y)}{\int \kepsnorm(x,y)\ud\mu(y)},
\label{eq:Teps}
\end{equation}
where 
$\kepsnorm(x,y) := \frac{\geps(x,y)}{\sqrt{\int \geps(x,y)
    \ud\mu(y)}\sqrt{\int \geps(x,y) \ud\mu(x)}}$ and
$\geps(x,y):=\frac{1}{(4\pi\epsilon)^\frac{d}{2}}\exp{(-\frac{|x-y|^2}{4\epsilon})}$
is the Gaussian kernel in $\mathbb{R}^d$.

In terms of the perturbed integral operator, the fixed-point
equation~\eqref{eq:fixed1} becomes,
\begin{equation}
\phieps = \Teps \phieps+ \int_0^\epsilon T^{(s)} h \ud s.
\label{eq:fixed2}
\end{equation}
The superscript $\epsilon$ is used to distinguish the approximate
($\epsilon$-dependent) solution from the exact solution $\phi$.  As
shown in the Appendix~\ref{app:app2}, $\Teps$ has an ergodic invariant measure
$\mueps$ which approximates $\mu$ as $\epsilon\downarrow 0$.
For any fixed positive $\epsilon$, we are interested in solutions that
are zero-mean with respect to this measure.       
The existence-uniqueness result for this solution is
described next; the proof appears in the Appendix~\ref{app:app2}.  

\medskip

\begin{proposition}
Consider the fixed-point problem~\eqref{eq:fixed2} with the perturbed
operator $\Teps$ defined according to~\eqref{eq:Teps}.  Fix
$\epsilon>0$.
Then there exists a unique solution $\phieps$ such that
$\int\phieps\ud\mueps=0$.  
\qed
\label{prop:Teps}
\end{proposition}

\medskip

In a numerical implementation, the solution $\phi^{\epsilon}$ is
approximated directly for the particles:
\begin{align*}
\Phi^{(\epsilon,N)}&:=(\phieps(X^1),\phieps (X^2),\cdots,\phieps (X^N)).
\end{align*}
The integral
operator $\Teps$ is approximated as a $N\times N$ Markov matrix whose $(i,j)$
entry is obtained empirically as,
\begin{equation}
\TepsN_{i,j} = \frac{\kepsnormN(X^i,X^j)}{\sum_{l=1}^N\kepsnormN(X^i,X^l)},
\label{eq:TepsN}
\end{equation}
where $\kepsnormN(x,y) =
\frac{g^\epsilon(x,y)}{\sqrt{\frac{1}{N}\sum_{l=1}^N
    g^\epsilon(x,X^l)}\sqrt{\frac{1}{N}\sum_{l=1}^N g^\epsilon(y,X^l)}}$.

\medskip

The resulting finite-dimensional fixed-point equation is given by,
\begin{equation}
\Phi^{(\epsilon,N)} = \TepsN \Phi^{(\epsilon,N)} + \int_0^{\epsilon}T^{(s,N)}H^{(N)} \ud s,
\label{eq:fixed-dis}
\end{equation}
where
$\Phi^{(\epsilon,N)}=(\Phi^{(\epsilon,N)}_1,\Phi^{(\epsilon,N)}_2,\hdots,\Phi^{(\epsilon,N)}_N)\in\Re^N$
is the vector-valued solution,
$H^{(N)}=(h(X^1),h(X^2),\cdots,h(X^N)) \in\Re^N$, and
$\TepsN$ and $T^{(s,N)}$ are $N\times N$ matrices defined according
to~\eqref{eq:TepsN}. 
The existence-uniqueness result for the zero-mean solution of the
finite-dimensional equation~\eqref{eq:fixed-dis} is
described next; its proof is given in the Appendix~\ref{app:finite_N}.  The zero-mean
property is the finite-dimensional counterpart of the zero-mean
condition $\int\phi\ud\mu=0$ for the original problem and
$\int\phi\ud\mu^\epsilon=0$ for the perturbed problem.

\medskip

\begin{proposition}
Consider the fixed-point problem~\eqref{eq:fixed-dis} with the matrix
$\TepsN$ defined according to~\eqref{eq:TepsN}.  Then, with
probability $1$, there exists a unique zero-mean solution
$\Phi^{(\epsilon,N)}\in\mathbb{R}^d$. 
\qed
\label{prop:TepsN}
\end{proposition}

\medskip

Once $\Phi^{(\epsilon,N)}$ is available, it is straightforward to extend it to
the entire domain.  For $x\in\mathbb{R}^d$,
\begin{equation}
\phiepsN(x) := 
\frac{\sum_{i=1}^N \kepsnormN(x,X^i)
  \Phi^{(\epsilon,N)}_i}{\sum_{i=1}^N\kepsnormN(x,X^i)} +
\int_0^{\epsilon}T^{(s,N)}h(x) \ud s.
\label{eqn:extension}
\end{equation}
By construction $\phiepsN(X^i) = \Phi^{(\epsilon,N)}_i$ for
$i=1,\hdots,N$. 
The extension is not necessary for filtering because one only needs to
solve for $\nabla\phi$ at $X^i$.  The formula for this is,
\begin{align*}
\frac{\partial \phi}{\partial x_l}(X^i) = \sum_{j=1}^N \left[\TepsN_{ij}\PhiepsN_j\left(X^j_l- \sum_{k=1}^N \TepsN_{ik}X^k_l\right)\right] 
\end{align*}
where $X^i_l \in \Re$ is the $l$-th component of $X^i \in \Re^d$. 

The algorithm is summarized in Algorithm \ref{alg:kernel}.  For
numerical purposes we use $\epsilon H^{N} \approx \int_0^\epsilon
T^{(s,N)}H^{(N)}\ud s$.  Also, the fixed-point
problem~\eqref{eq:fixed-dis} is conveniently solved using the method
of successive approximation.  In filtering, the initial guess is
readily available from the solution at the previous time-step.  

\begin{algorithm}
\caption{Kernel-based gain function approximation}
\begin{algorithmic}
\REQUIRE $\{X^i\}_{i=1}^N$, $H:=\{h(X^i)\}_{i=1}^N$, $\epsilon$ 
\ENSURE $\Phi:=\{\phi(X^i)\}_{i=1}^N$, $\{\nabla \phi(X^i)\}_{i=1}^N$ \medskip
\STATE Calculate $g_{ij}:=\exp(-|X^i-X^j|^2/4\epsilon)$ for $i,j=1$ to $N$.\medskip
\STATE Calculate $k_{ij}:=\frac{g_{ij}}{\sqrt{\sum_l g_{il}}\sqrt{\sum_l g_{jl}}}$ for $i,j=1$ to $N$.
\STATE Calculate $T_{ij}:=\frac{k_{ij}}{\sum_l k_{il}}$ for $i,j=1$ to $N$. \medskip
\STATE Solve $\Phi= T \Phi + \epsilon H$ for $\Phi$ (Successive approximation). \medskip
\STATE Calculate $\frac{\partial \phi}{\partial x_l}(X^i) = \sum_{j=1}^N \left[T_{ij}\Phi_j\left(X^j_l- \sum_{k=1}^N T_{ik}X^k_l\right)\right]$ for $l=1$ to $d$.
\end{algorithmic}
\label{alg:kernel}
\end{algorithm}

\newP{Convergence} 
The following is a summary of the approximations with the kernel-based method:
\begin{align*}
\text{Exact}:&~&\phi &= e^{\epsilon\Delta_\rho} \phi+ \int_0^\epsilon e^{s\Delta_\rho}h \ud s \\
\text{Kernel approx:}&~ &\phieps &= \Teps \phieps+ \int_0^\epsilon T^{(s)} h \ud s\\
\text{Empirical approx:}&~ &\phiepsN &= \TepsN \phiepsN + \int_0^{\epsilon}T^{(s,N)}h \ud s
\end{align*}
We break the convergence analysis into two steps. The first step
involves convergence of  $\phiepsN$ to $\phieps$ as $N \to
\infty$. The second step involves  convergence of $\phieps$  to $\phi$
as $\epsilon \to 0$.  The following theorem states the convergence
result for the first step; a sketch of the proof appears in
Appendix~\ref{app:conv_kernel}. 

\medskip
\begin{theorem}
Consider the empirical kernel approximation of the fixed-point equation~\eqref{eq:fixed1}.  Fix $\epsilon>0$.  Then, 
\begin{romannum}
\item There exists a unique (zero-mean) solution $\phieps$ for the perturbed
  fixed-point equation~\eqref{eq:fixed2}.
\item For any finite $N$, a unique (zero-mean) solution
  $\Phi^{(\epsilon,N)}$ for~\eqref{eq:fixed-dis} exists with
  probability 1. 
\end{romannum}
For a compact set $\Omega\subset \Re^d$,   
\begin{equation}
\lim_{N \to \infty}~ \sup_{x \in \Omega} | \phiepsN(x) - \phieps(x)| =
0,\quad \text{a.s}, 
\label{eq:conv_kernel}
\end{equation}
where $\phiepsN$ is the extension of the vector-valued solution
$\Phi^{(\epsilon,N)}$ (see~\eqref{eqn:extension}).\qed
\label{thm:TepsN}
\end{theorem}

\medskip

The convergence analysis for step 2, as $\epsilon \to 0$, is the
subject of ongoing work.  In this regard, it is shown
in~\cite{hein-consistency-2005} that for compactly supported functions $f \in C^3$,
\begin{equation*}
f(x) - \Teps f(x)= \epsilon\, \Delta_\rho f(x) + O(\epsilon^2).
\end{equation*}

\medskip

\begin{example}
Consider once more the bimodal distribution introduced
in~\ref{example:truesol}.  Figure~\ref{fig:GLsol} depicts the
kernel-based approximation of the gain function with $N=200$
particles and range of $\epsilon$.  The kernel-based avoids the Gibbs
phenomena observed with the Galerkin (see \Fig{fig:galerkin}).
Notably, the gain function is positive for any choice of $\epsilon$.     \qed
   
\begin{figure}
\centering
\includegraphics[width=\columnwidth]{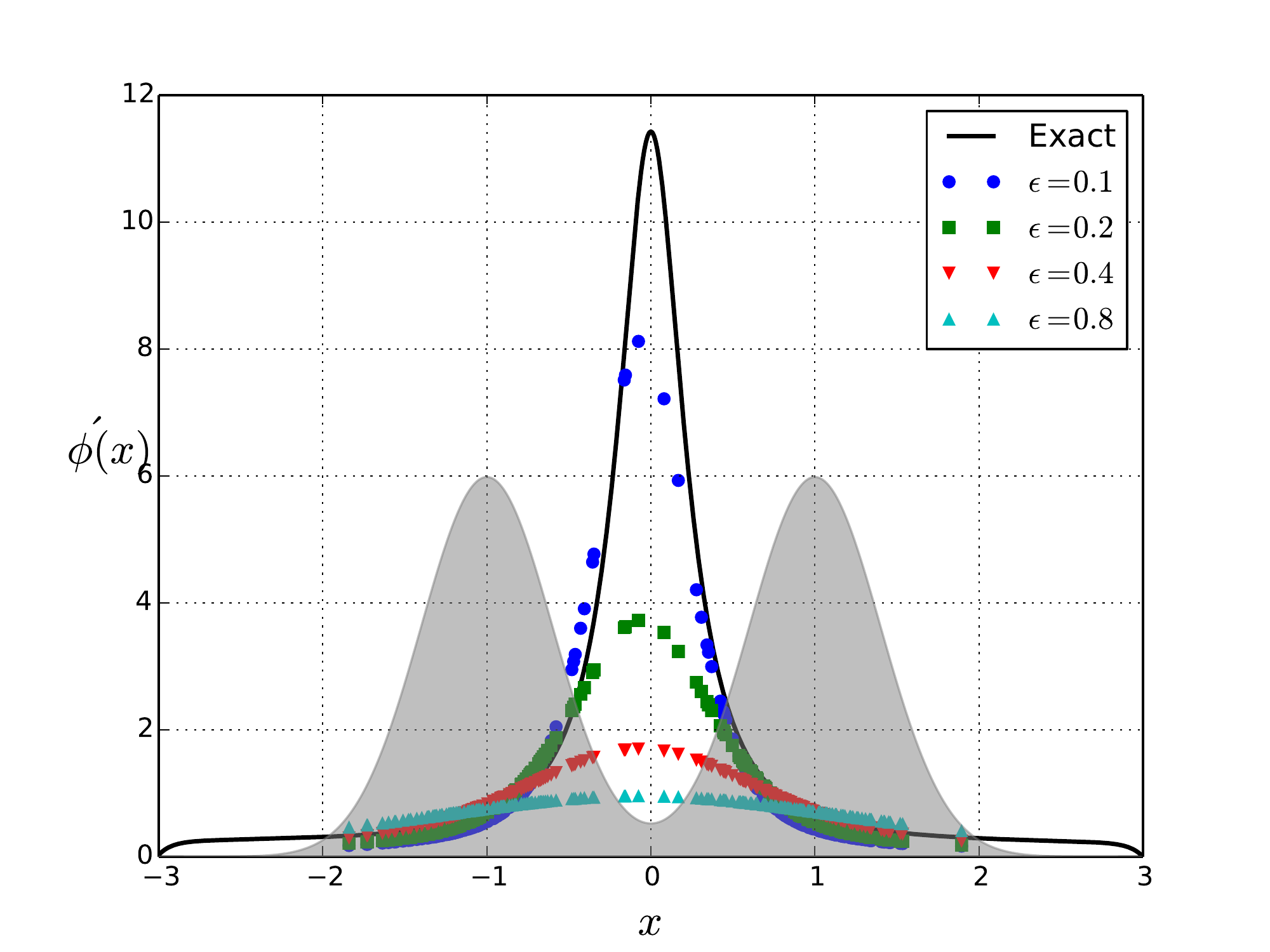}
\caption{Comparison of the exact solution and its kernel-based
  approximation with $\epsilon=0.1,0.2,0.4,0.8$ and $N=200$
  particles.  The density is depicted as the shaded curve in the background.}
\label{fig:GLsol}
\end{figure}
\end{example}

\begin{figure*}[t]
\centering
\begin{tabular}{cc}
\subfigure[]{ \includegraphics[width=1.0\columnwidth]{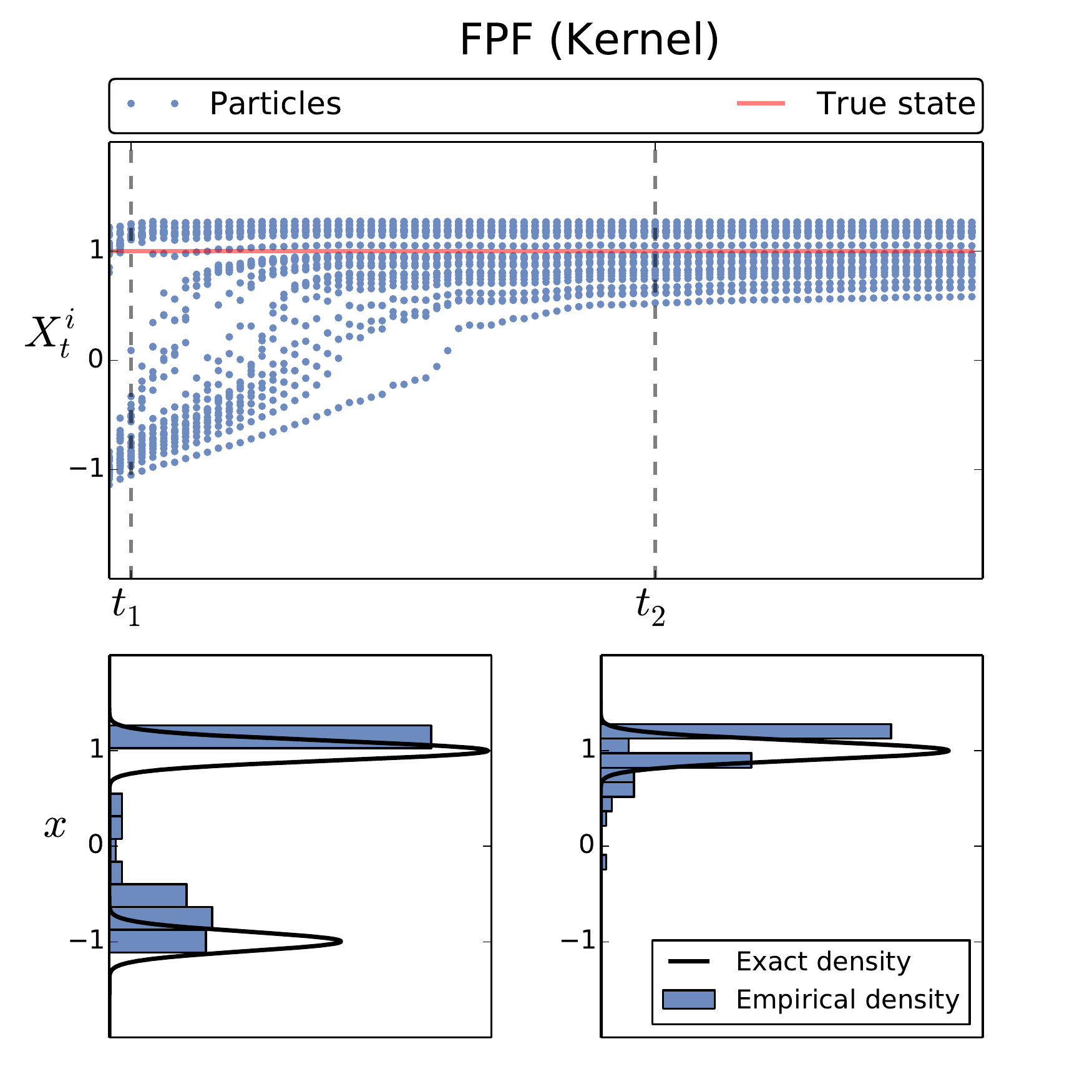}
\label{fig:Xi-Graph}} &
\subfigure[]{\includegraphics[width=1.0\columnwidth]{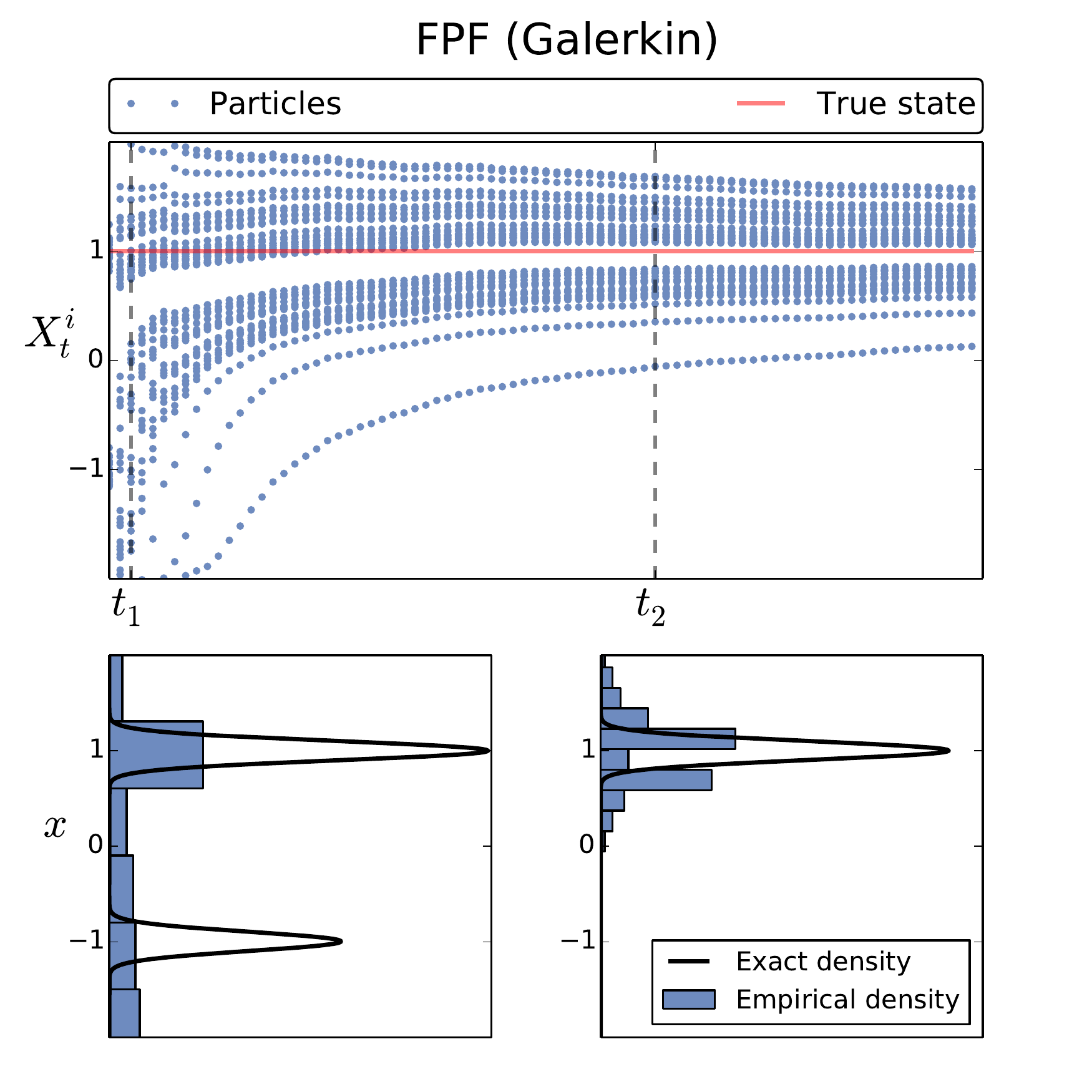}
\label{fig:Xi-Galerkin}} 
\end{tabular}
\caption{Comparison of simulation results: Trajectory of particles and posterior distribution with (a) Kernel-based approximation of the gain function, (b) Galerkin approximation of the gain function.}
\label{fig:sim-result}
\end{figure*}
\section{Numerics}
\label{sec:numerics}

In this section, we consider a filtering problem associated with the
bimodal distribution introduced in Example~\ref{example:truesol}.  
The filtering model is given by,  
\begin{align*}
\ud X_t &= 0,\\
\ud Z_t &= X_t\ud t + \sigma_w \ud W_t,
\end{align*} 
where $X_t \in \Re$, $Z_t \in \Re$, $W_t$ is a standard Wiener
process, the initial condition $X_0$ is sampled from the bimodal
distribution comprising of two Gaussians, $N(-1,\sigma^2)$ and
$N(+1,\sigma^2)$, and without loss of generality, $Z_0=0$.  As in Example~\ref{example:truesol}, the
observation function $h(x)=x$ is linear.  The static case is
considered because the posterior is given explicitly:
\begin{equation}
p^*(x,t) \doteq (\text{const.}) \exp \left( \frac{1}{2\sigma_W^2} h(x) \, Z_t - \frac{1}{4\sigma_W^2}
|h(x)|^2 \, t \right) p_0^*(x).
\label{eq:explicit-posterior}
\end{equation}          

The following filtering algorithms are implemented for comparison:
\begin{enumerate}
\item Kalman filter;
\item Feedback particle filter with the Galerkin approximation where
  $S=\text{span}\{x,x^2,x^3,x^4,x^5\}$;
\item Feedback particle filter with the kernel approximation.  
\end{enumerate}


The performance metrics are as follows:
\begin{enumerate}
\item Filter mean $\hat{X}_t$;
\item Conditional probability
  $\PP\left[|X_t-X_0|<\frac{1}{2}|\clZ_t\right]$.
\end{enumerate}

\medskip

The simulation parameters are as follows: The true initial state
$X_0=1$.  The measurement noise parameter $\sigma_W=0.3$.  The
simulation is carried out over a finite time-horizon $t\in[0,T]$ with
$T=0.8$ and a fixed discretized time-step $\Delta t = 0.02$.  All the
particle filters use $N=100$ particles and have the same
initialization, where particles are drawn with equal probability from
one of the two Gaussians, $N(-1,\sigma^2)$ or
$N(+1,\sigma^2)$, where $\sigma=0.1$.  For the kernel approximation,
we use $\epsilon=0.15$.  The simulation parameters are tabulated in
Table~I.  The Kalman filter is initialized with $\hat{X}_0 = 0 $ and
$\Sigma_0 =\text{Var}(X_0)=1+\sigma^2 $.  The latter corresponds to the variance of the prior.  
Figure~\ref{fig:sim-result}~parts (a) and (b) depict the particle
trajectories and the associated distributions obtained using the
kernel approximation and the Galerkin approximation, respectively.
The kernel-based approximation provides for a better approximation of
the exact posterior.  At time $t_1$ during the initial transients, some of the particles with the
Galerkin approximation show a divergence.  This is a numerical issue
due to the Gibb's phenomena that leads to erroneous negative value of
the gain (see the discussion in Examples~\ref{example:truesol}
and~~\ref{example:Galerkin}).

Figure~\ref{fig:sim-result-2} depicts a comparison of the simulation
results for the two metrics.  For the particle filters, these are
computed empirically. 

For applications of FPF with kernel-based approximation of the gain function for attitude estimation problem see the companion paper~\cite{ChiCDC2016}. 
\begin{figure*}[t]
\centering
\begin{tabular}{cc}
\subfigure[]{
\includegraphics[width=1.0\columnwidth]{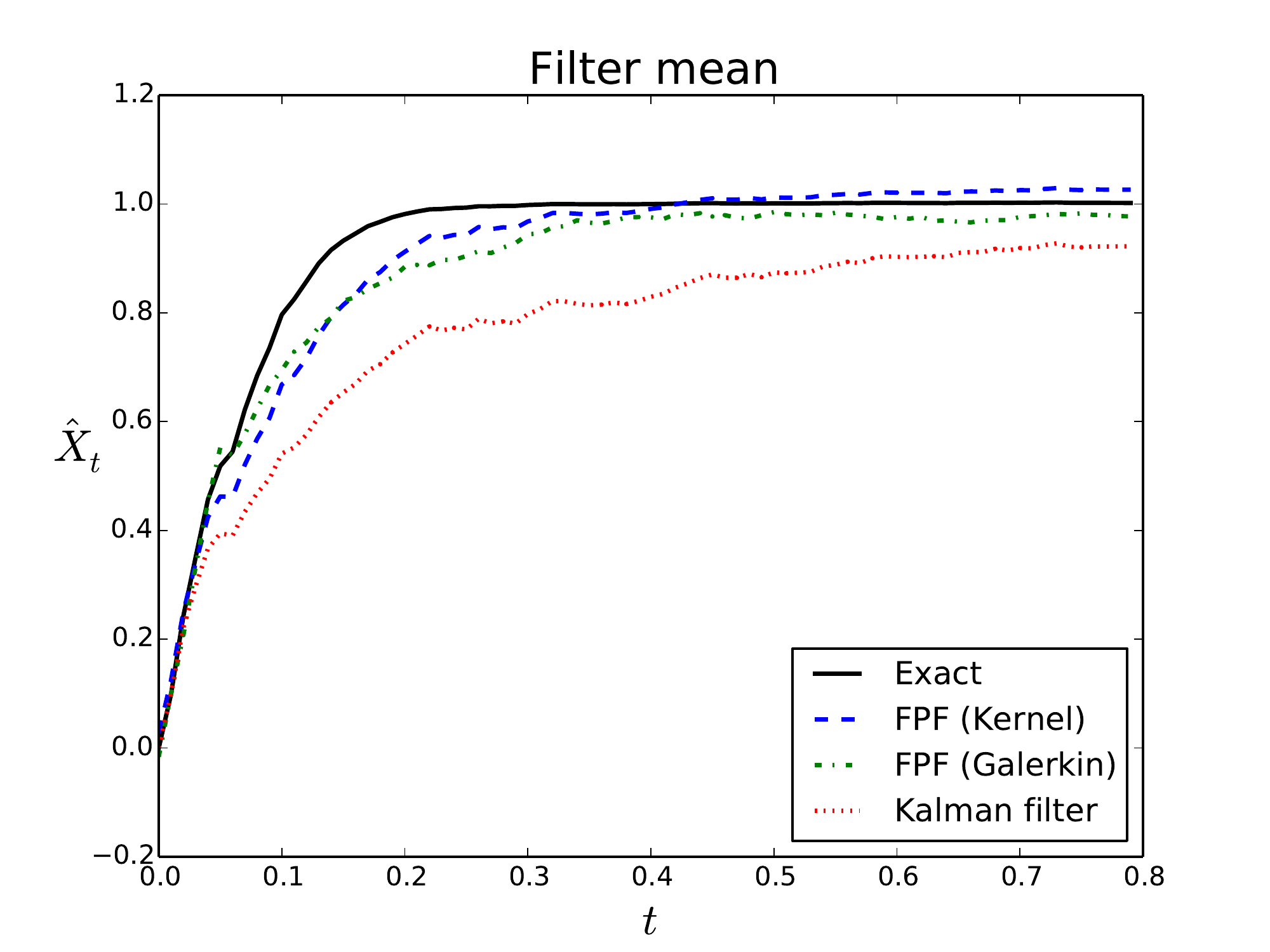}
\label{fig:mean}
} &
\subfigure[]{
\includegraphics[width=1.0\columnwidth]{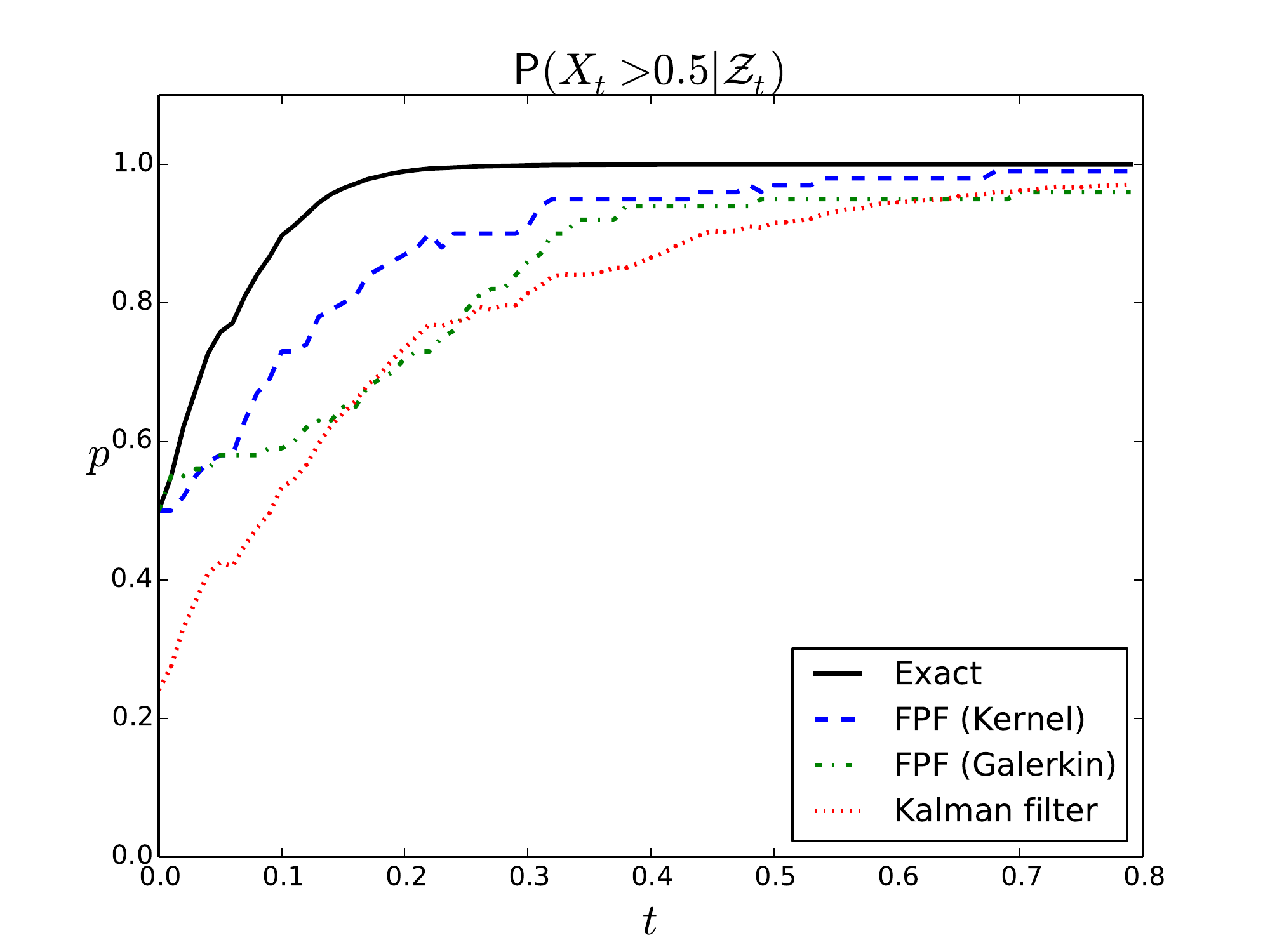}
\label{fig:pXgeq}
}
\end{tabular}
\caption{Comparison of simulation results: (a) Conditional mean of the state $X_t$,  (b) Conditional probability of $X_t$ larger than $\frac{1}{2}$.}
\label{fig:sim-result-2}
\end{figure*}
\bibliographystyle{plain}
\bibliography{fpfbib,ref}
\appendix
\subsection{Proof of Proposition \ref{prop:galerkin} }
\label{app:app1}

Using the spectral representation~\eqref{eq:spectral_rep}, because $h\in L^2$,
\[
\phi = - \Delta_{\rho}^{-1} h = \sum_{m=1}^{\infty} \frac{1}{\lambda_m}<e_m,h>e_m.
\]
With basis functions as eigenfunctions,
\[
\phi^{(M)} = \sum_{m=1}^{M} \frac{1}{\lambda_m}<e_m,h>e_m.
\]
Therefore,
\[
\| \phi - \phi^{(M)}\|^2_2 = \sum_{m=M+1}^\infty \frac{1}{\lambda_m^2}
|<e_m,h>|^2 \le \frac{1}{\lambda_M^2} \|h\|^2_2.
\]

\medskip

Next, by the triangle inequality,
\[
\|\phi-\phi^{(M,N)}\|_{2} \le \|\phi-\phi^{(M)}\|_2 + \|\phi^{(M)}-\phi^{(M,N)}\|_{2}. 
\]
We show $\|\phi^{(M)}-\phi^{(M,N)}\|_{2}\rightarrow 0$ a.s. as
$N\rightarrow\infty$.  

Using the formulae~\eqref{eq:phiM} and~\eqref{eq:phiMN} with
basis-functions $\psi_m$ as eigenfunctions $e_m$,
\[
 \|\phi^{(M)}-\phi^{(M,N)}\|_{2} = |c-c^{N}|,
\]
where $|\cdot|$ denotes the Euclidean norm in $\Re^M$.  The vectors
$c$ and $c^{(N)}$ solve the matrix equations
(see~\eqref{eqn:linmatrixeqn} and~\eqref{eqn:linmatrixeqnN}),
\begin{align*}
Ac &= b,\\
A^{(N)} c^{(N)} &= b^{(N)},
\end{align*}
where $A^{(N)} \stackrel{\text{a.s.}}{\longrightarrow} A$, $b^{(N)}
\stackrel{\text{a.s.}}{\longrightarrow} b$ by the strong law of large
numbers.  Consequently, because
$A=\text{diag}(\lambda_1,\ldots,\lambda_N)$ is invertible,
$c^{(N)}\rightarrow c$ a.s. as $N\rightarrow\infty$. 

\subsection{Proof of Proposition \ref{prop:Teps}}
\label{app:app2}
Denote $\neps(x):=\int \keps(x,y)\ud \mu(y)$, and re-write the operator $\Teps$ as,
\begin{equation*}
\Teps f(x) = \int \frac{\keps(x,y)}{\neps(x)\neps(y)} \, f(y) \; \ud \mueps(y),
\end{equation*}
where $\ud \mueps (x) := \neps(x) \ud \mu(x)$.

Denote $L^2(\mueps)$ as the space of square integrable functions with
respect to $\mueps$ and as before $L^2_0(\mueps) :=\big\{ \phi \in
L^2(\mueps) \big| \int \phi \ud \mueps=0\big\}$ is the co-dimension
$1$ subspace of mean-zero functions in $L^2(\mueps)$.  
The technical part of proving the Proposition is to show that the
operator $\Teps$ is a strict contraction on the subspace.  
 
\begin{lemma}
Suppose $\rho$, the density of the probability measure $\mu$,
satisfies Assumption A1. Then,
\begin{romannum}
\item $\mueps$ is a finite measure.
\item For sufficiently small values of $\epsilon$, the operator $\Teps:L^2(\mueps) \to
  L^2(\mueps)$ is a compact Markov operator with an invariant measure
  $\mueps$.
\item $\Teps:\Ltwoeps \to \Ltwoeps$ is a strict contraction.
\end{romannum}
\label{lemma:Teps}
\end{lemma}
\begin{proof}

\noindent (i) WLOG assume $\mu=0$ in the Assumption A1. 
For notational ease, denote 
\[
\rhoeps(x):=\int
  \geps(x,y)\rho(y)\ud y,
\] 
where recall that $\rhoeps(x)$ is used to define the
  denominator of the kernel.  Then
\[
c_1 \exp(-\half x^T Q_1^{-1}x \; \leq\; \rhoeps(x) \; \leq \; c_2 \exp(-\half x^T
 Q_1^{-1}x),
\]
where $Q_1^{-1}:=(\Sigma + 2\epsilon I)^{-1}$ and $c_1=(2\pi)^{-\frac{d}{2}}|Q_1|^{-\half}e^{-\|V\|_\infty}$ and $c_2=(2\pi)^{-\frac{d}{2}}|Q_1|^{-\half}e^{\|V\|_\infty}$ are positive
constants that depend on $\|V\|_{\infty}$. 
Therefore,
\begin{align*}
\mueps (\Re^d)  &= \int \int \keps(x,y) \ud \mu(x) \ud \mu(y) \\
&= \int \int \frac{\geps(x,y)}{\sqrt{\rhoeps(x)}\sqrt{\rhoeps(y)}} \rho(x) \rho(y) \ud x \ud y\\
&\leq  \frac{1}{c_1} \int \int \geps(x,y) \, e^{-\half x^TQ_2^{-1}x} \, e^{-\half
  y^TQ_2^{-1}y} \; \ud x \ud y,
\end{align*}
which is bounded because $Q_2^{-1} := \Sigma^{-1} - \half Q_1^{-1}
\succ 0$. 
This proves that $\mueps$ is a finite measure. 

\noindent (ii) The integral operator $T^{\epsilon}$ is a Markov operator because the
kernel $k^{(\epsilon)}(x,y)>0$ and 
\[
\Teps 1 = \int \frac{\keps(x,y)}{\neps(x)\neps(y)}\ud \mueps(y) = \int \frac{\keps(x,y)}{\neps(x)}\ud \mu(y) = 1.
\]
$T^{\epsilon}$ is compact because [Theorem 7.2.7 in~\cite{hutson2005}],
\[
\int\int \left|\frac{{\keps}(x,y)}{{\neps}(x){\neps}(y)}\right|^{2} \ud
\mueps(x) \ud \mueps(y) < \infty.
\]
The inequality holds because the integrand can be bounded by a
Gaussian:
\[
\exp( - \frac{|x-y|^2}{2\epsilon} - \half x^TQ_3^{-1}x^T - \half y^TQ_3 y)
\]
where $Q_3^{-1}:=\Sigma^{-1} - \half Q_1^{-1}-(Q_2 + 2\epsilon I)^{-1}$ is of order $O(\epsilon)$.

Finally, the measure $\mu^{\epsilon}$ is an invariant measure because
for all functions $f$, 
\[
\int \Teps f(x) \; \ud \mueps (x) = \int f(y)\; \ud \mueps(y).
\]


\noindent (iii) Since $\mu^{(\epsilon)}$ is an invariant measure,
$\Teps:\Ltwoeps \to \Ltwoeps$.  Next, for all $f,\,g \in \Ltwoeps$,
\begin{align*}
&\half\int f(x)^2 \ud \mueps(x) + \half
\int g(y)^2 \ud \mueps(y) \\&- \int \int
\frac{\keps(x,y)}{\neps(x)\neps(y)} f(x)g(y)\ud \mueps(x) \ud
\mueps(y) \\
&=\int \int \frac{\keps(x,y)}{\neps(x)\neps(y)} (f(x)-g(y))^2 \ud
\mueps(x) \ud \mueps(y) \\&\geq 0,
\end{align*}
where, because the kernel is everywhere positive, the equality holds
iff
\begin{equation*}
f(x)=g(y)=\text{(const.)}\quad \mu^{\epsilon}-\text{a.e.}
\end{equation*}
 Therefore, substituting $g=\Teps f$ in the inequality above,
\begin{equation*}
\|\Teps f\|^2 \leq \frac{1}{2} \|f\|^2 + \frac{1}{2}\|\Teps f\|^2 \Rightarrow \\
\|\Teps f\| \leq \|f\|,
\end{equation*}
where the $L^2$ norms here are with respect to the invariant measure
$\mu^{\epsilon}$.  Now, for $f\in\Ltwoeps$, $\int f \ud \mueps =0$,
and thus the equality holds iff
\begin{equation*}
f(x)=\text{(const.)}=0.
\end{equation*}
Therefore, $\Teps$ is strictly
contractive on $\Ltwoeps$.
\end{proof}

\medskip

The proof of the Prop.~\ref{prop:Teps} now follows because $\Teps$ is a
contraction on  $\Ltwoeps$.  Note also that $\mu^{(\epsilon)}(A)
\rightarrow \mu(A)$ for any measurable set, as $\epsilon\downarrow
0$.  

\subsection{Proof of Proposition \ref{prop:TepsN}}
\label{app:finite_N}
By construction, $\TepsN$ is a $N \times N$ stochastic matrix whose
entries are all positive with probability $1$.  The result follows.
\subsection{Proof sketch for Theorem \ref{thm:TepsN}}
\label{app:conv_kernel}
Parts (i) and (ii) have already been proved as part of the
Prop.~\ref{prop:Teps} and Prop.~\ref{prop:TepsN}, respectively.
The convergence result~\eqref{eq:conv_kernel} leans on approximation
theory for integral operators. In [Theorem 7.6.6
in~\cite{hutson2005}], it is shown that if
\begin{romannum}
\item $\Teps$ is compact and $\TepsN$ is collectively compact;
\item $\|\Teps\|<1$;
\item $\TepsN$ converges to $\Teps$ pointwise, i.e
\begin{equation*}
\lim_{N \to \infty}\|\TepsN f - \Teps f\|_\infty =
0\quad\text{a.s},\quad \forall f 
\end{equation*} 
\end{romannum}
Then for $N$ large enough, $(I-\TepsN)^{-1}$ is bounded and 
\begin{equation*}
\lim_{N \to \infty}\|(I-\TepsN)^{-1} h  - (I-\Teps)^{-1}h \|_\infty = 0,\quad\text{a.s}\quad \forall h.
\end{equation*}

In order to prove the convergence result, we consider the Banach space
$\FF:=\{f \in C(\Omega); ~\int f \ud \mu(x)=0\}$ equipped with the
$\|\cdot\|_\infty$ norm.  We consider $\Teps$ and $\TepsN$ as linear
operators on $\FF$.  Note that $\TepsN$ here corresponds to the
extension of the vector-valued solution to $\Re^d$ using~\eqref{eqn:extension}.

The proof involves verification of each of the three requirements stated
above:
\begin{romannum}
\item Collective compactness follows because $\keps$ is continuous
  [Theorem 7.2.6~\cite{hutson2005}].
\item The norm condition holds because 
\begin{align*}
|\Teps f(x) | &\leq \frac{\int \kepsnorm(x,y)|f(y)|\ud \mu(y)}{\int \kepsnorm(x,y)\ud \mu(y)} \\
&\leq \frac{\int \kepsnorm(x,y)\ud \mu(y)}{\int \kepsnorm(x,y)\ud \mu(y)} \infnorm{f}\leq \infnorm{f}
\end{align*} 
where the equality holds only when $f$ is constant.  For $f\in\FF$,
this constant can only be $0$. 
\item Pointwise convergence follows from using the LLN.  For a fixed
  continuous function $f$, LLN implies convergence for every fixed
  $x\in\Re^d$.  On a compact set $\Omega$, pointwise convergence
  also implies uniform convergence.  
\end{romannum}
\qed



\end{document}